\documentclass[11pt]{article}
\usepackage[margin=1in]{geometry}
\usepackage{amsthm,amsmath}
\usepackage{latexsym,amssymb,amsmath}
\usepackage{graphicx,float,color,fancybox,shapepar,setspace,hyperref}
\usepackage{subfigure}
\usepackage{pgf,tikz}
\usepackage[affil-it]{authblk}
\usepackage{indentfirst}
\usepackage{hyperref}
\usepackage[section]{placeins}
\hypersetup
{
colorlinks=true,
linkcolor=blue,
filecolor=blue,
urlcolor=blue,
citecolor=cyan,
}
\usetikzlibrary{arrows}
\newtheorem{thm}{Theorem}[section]

\newtheorem{defn}[thm]{Definition}

\def\~{\sim}

\def\sf{\mathsf}

\newcommand{\FS}{\mathsf{FS}}

\newtheorem{theorem}{Theorem}[section]
\newtheorem{proposition}[theorem]{Proposition}
\newtheorem{corollary}[theorem]{Corollary}

\newtheorem{problem}[theorem]{Problem}
\newtheorem{lemma}[theorem]{Lemma}

\newcommand{\Spider}{\mathsf{Spider}}

\begin{document}

\title{The connectedness of the friends-and-strangers graph of a lollipop and others}

\author{Lanchao WANG,  Yaojun CHEN\thanks{Corresponding author. 
 Email: 
\href{mailto://yaojunc@nju.edu.cn}{ yaojunc@nju.edu.cn}.
}}
 \affil{ { \small {Department of Mathematics, Nanjing University, Nanjing 210093, China}}
 }
\date{}
\maketitle

\begin{abstract}
Let $X$ and $Y$ be any two graphs of order $n$. The friends-and-strangers graph $\mathsf{FS}(X,Y)$ of $X$ and $Y$ is a graph with vertex set consisting of all bijections $\sigma :V(X) \mapsto V(Y)$, in which two bijections $\sigma$, $\sigma'$ are adjacent if and only if they differ precisely on two adjacent vertices of $X$, and the corresponding mappings are adjacent in $Y$. The most fundamental question that one can ask about these friends-and-strangers graphs is whether or not they are connected.  Let $\mathsf{Lollipop}_{n-k,k}$ be a lollipop graph of order $n$ obtained by  identifying one end of a path of order $n-k+1$ with a vertex of a complete graph of order $k$. Defant and Kravitz started to study the connectedness of $\FS(\mathsf{Lollipop}_{n-k,k},Y)$. In this paper, we give a sufficient and necessary condition for $\FS(\mathsf{Lollipop}_{n-k,k},Y)$ to be connected  for all $2\leq k\leq n$.


\end{abstract}

\section{Introduction}

 All graphs considered in this paper are simple without loops. Let $G=(V(G),E(G))$ be a graph. We use $\delta(G)$ and $\Delta(G)$  to denote its minimum and maximum degree, respectively. For $S\subseteq V(G)$, $G|_S$ denotes the subgraph of $G$ induced by $S$. Let $\sf{Path}_n$, $\sf{Cycle}_n$, $\sf{Star}_n$ and $\sf{Complete}_n$ denote a path, a cycle, a star and a complete graph of order $n$, respectively.   Let $\mathsf{Lollipop}_{n-k,k}$ be a lollipop graph of order $n$ obtained by  identifying one end of a $\sf{Path}_{n-k+1}$ with a vertex of a $\sf{Complete}_k$. A graph $G$ is $\ell$-connected if the resulting graph is still connected by removing any $\ell-1$ vertices from $G$.
For two graphs $G$ and $H$, we let $G\cup H$ denote the disjoint union of $G$ and $H$.

The friends-and-strangers graphs were introduced by Defant and Kravitz \cite{DK}, which are defined as follows. 

 \begin{defn} 
Let $X$ and $Y$ be two graphs, each with $n$ vertices. The friends-and-strangers graph $\mathsf{FS}(X,Y)$ of $X$ and $Y$ is a graph with vertex set consisting of all bijections from $V(X)$ to $V(Y)$, two such bijections $\sigma$, $\sigma'$ are adjacent if and only if they differ precisely on two adjacent vertices, say $a,b\in X$ with $\{a,b\}\in E(X)$, and the corresponding mappings are adjacent in $Y$, i.e.
\begin{itemize}
\item  $\{\sigma(a), \sigma(b)\} \in E(Y)$;
\item  $\sigma(a)=\sigma'(b)$, $\sigma(b)=\sigma'(a)$ and $\sigma(c)=\sigma'(c)$ for all $c\in V(X)\backslash \{a,b\}.$
\end{itemize}
\end{defn}

The friends-and-strangers graph $\mathsf{FS}(X,Y)$ can be interpreted as follows.  View $V(X)$ as $n$ positions and $V(Y)$ as $n$ people.  Two people are friends if and only if they are adjacent in $Y$ and two positions are adjacent if and only if they are adjacent in $X$. A bijection from $V(X)$ to $V(Y)$ represents $n$ people standing on these $n$ positions such that each person stands on precisely one position.
 At any point of time, two people can swap their positions if and only if they are friends and the two positions they stand are adjacent. A natural question is  how various configurations can be reached from other configurations when  multiple such swaps are allowed. This is precisely the information that is encoded in $\mathsf{FS}(X,Y)$. Note that the components of $\mathsf{FS}(X,Y)$ are the equivalence classes of mutually-reachable (by the multiple swaps described above) configurations, so the connectivity,  is the basic aspect  of interest in  friends-and-strangers graphs. 


The questions and results in literature on the friends-and-strangers graph $\mathsf{FS}(X,Y)$ roughly fall in three types
\begin{itemize}

\item The structure of $\mathsf{FS}(X,Y)$ when at least one of $X,Y$ are specific graphs, such as paths, cycles, lollipop graphs, spider graphs and so on, \cite{DDLW}, \cite{DK}, \cite{L}, \cite{J1}, \cite{W}.  
\item The structure  of $\FS(X,Y)$ when none of $X,Y$ is specific graph, such as minimum degree conditions on $X$ and $Y$, the case when $X$ has a Hamiltonian path, the non-polynomially bounded diameters and so on,  \cite{ADK}-\cite{DK}, \cite{J1}, \cite{J2}.
\item The structure of $\mathsf{FS}(X,Y)$ when both $X$ and $Y$ are random graphs, \cite{ADK}, \cite{M}, \cite{Wang}. 

 \end{itemize}

 We note that Milojevic \cite{M} also studied a new model of friends-and-strangers graphs.

 The structure of $\mathsf{FS}(X,Y)$ when  $X,Y$ belong to the first type is a basic question on the topic related to friends-and-strangers graphs, and the results on this type can also be used to study the other two types. For example, Alon, Defant and Kravitz \cite{ADK} used the structure of $\FS(\sf{Star}_n,Y)$ and $\FS(\sf{Lollipop}_{n-3,3},\sf{Star}_n)$  in  researching  the random aspect of friends-and-strangers graphs and minimum degree conditions on $X, Y$ for the connectedness of $\FS(X,Y)$, respectively;  Jeong \cite{J2} used the structure of 
 $\FS(\sf{Cycle}_n,Y)$ to investigate the connectedness of $\FS(X,Y)$ when $X$ is 2-connected. 
 
Fix $X=\sf{Path}_n$ or $\sf{Complete}_n$, the connectedness of $\FS(X,Y)$ is characterized as follows.
\begin{theorem}\cite{DK}\label{path}
 Let $Y$ be a graph on $n$ vertices. Then  $\FS(\sf{Path}_n,Y)$ is connected if and only if $Y$ is complete.
\end{theorem}
\begin{theorem}\cite{GR} \label{complete}
 Let $Y$ be a graph on $n$ vertices. Then  $\FS(\sf{Complete}_n,Y)$ is connected if and only if $Y$ is connected.
\end{theorem}

Defant and Kravitz \cite{DK} started to consider the connectedness of $\mathsf{FS}(\sf{Lollipop}_{n-k,k},Y)$. They first established a necessary condition in terms of $\delta(Y)$ for $\mathsf{FS}(\sf{Lollipop}_{n-k,k},Y)$ to be connected.

\begin{theorem}\cite{DK}\label{t1}
 Let $Y$ be a graph on $n$ vertices.  If $\FS(\sf{Lollipop}_{n-k,k},Y)$ is connected, then $\delta(Y)\ge n-k+1$.
\end{theorem}
\noindent Moreover, they tried to characterize the connectedness of $\FS(\sf{Lollipop}_{n-k,k},Y)$ in the case when $k=3,5$, and obtained the following.
\begin{theorem}\cite{DK}\label{L1}
Let $Y$ be a graph on $n$ vertices. Then $\FS(\sf{Lollipop}_{n-3,3},Y)$ is connected if and only if $\delta(Y)\ge n-2$.
\end{theorem}
\begin{theorem}\cite{DK}\label{L2}
Let $Y$ be a graph on $n$ vertices. Then $\FS(\sf{Lollipop}_{n-5,5},Y)$ is connected if $\delta(Y)\ge n-3$,  and $\FS(\sf{Lollipop}_{n-5,5},Y)$ is disconnected if $\delta(Y)\le n-5$.
\end{theorem}

Obviously,  the connectedness of $\FS(\sf{Lollipop}_{n-3,3},Y)$ is completely determined by Theorem \ref{L1}. However,  there is a ``gap''  between the lower and upper bounds of $\delta(Y)$ that guarantee $\FS(\sf{Lollipop}_{n-5,5},Y)$ to be connected or not by Theorem \ref{L2}. So, Defant and Kravitz \cite{DK}  raised a question: Characterize the graphs $Y$ with $\delta(Y)=n-4$ for which  $\FS(\mathsf{Lollipop}_{n-5,5},Y)$ is connected. 
\vskip 2mm
In this paper, we give a sufficient and necessary condition for $\FS(\mathsf{Lollipop}_{n-k,k},Y)$ to be connected for all $2\leq k\leq n$, and the main result is as below.

 \begin{theorem}\label{main}  Let $2\le k \le n$ be integers and $Y$ be a graph on $n$ vertices. Then the graph $\FS(\sf{Lollipop}_{n-k,k},Y)$ is connected if and only if every $k$-vertex induced subgraph of $Y$ is connected, which is equivalent to $Y$ is $(n-k+1)$-connected.
\end{theorem}

One can see that Theorem  \ref{main} strengthens Theorem \ref{t1} and extends Theorems \ref{path} and \ref{complete}. In addition, by Theorem  \ref{main}, we can easily deduce the following.

\begin{corollary}\label{co}
 Let $Y$ be a  graph on $n$ vertices with $\delta(Y)=n-4$.  Then the graph $\FS(\mathsf{Lollipop}_{n-5,5},Y)$ is connected if and only if  $Y$ does not contain any induced subgraph isomorphic to $\sf{Complete}_3\cup \sf{Path}_2$ or $\sf{Path}_3\cup \sf{Path}_2$.
\end{corollary}

It is clear that Corollary \ref{co} characterizes the graphs $Y$ with $\delta(Y)=n-4$ for which  $\FS(\mathsf{Lollipop}_{n-5,5},Y)$ is connected. On the other hand, it is not difficult to check that there are many graphs $Y$ on $n$ vertices with $\delta(Y)=n-4$ which do/don't contain $\sf{Complete}_3\cup \sf{Path}_2$ or $\sf{Path}_3\cup \sf{Path}_2$ as an induced subgraph.

%
%

\section{Proof of Theorem \ref{main}}

In order to prove Theorem \ref{main}, we assume that the $\sf{Path}_{n-k+1}$ in $\mathsf{Lollipop}_{n-k,k}$ has vertex set $[n-k+1]$ and edge set 
$\{ \{1,2\}, \{2,3\},\dots, \{n-k,n-k+1\} \}$ and the $\sf{Complete}_k$ is on the set  $\{n-k+1,\dots,n\}$, where $2\le k \le n$.

We divide the proof of Theorem \ref{main} into two parts: connected and disconnected.


\subsection{Connectedness}
\begin{proposition}\label{p1}
Let $2\le k \le n$ be integers and $Y$ be a graph on $n$ vertices. If every $k$-vertices induced subgraph of $Y$ is connected, then $\FS(\sf{Lollipop}_{n-k,k},Y)$ is connected.
\end{proposition}

We need the following result due to Defant, Dong, Lee and Wei \cite{DDLW}.
\begin{lemma}\cite{DDLW} \label{any}
Let $X$ and $Y$ be connected  graphs on $n$ vertices with $\Delta(X)=k\geq 2$. Suppose every  induced subgraph of $Y$ with $k$ vertices is connected. Let $\sigma$ be a vertex of $\FS(X,Y)$, and fix $x\in V(X)$ and $y\in V(Y)$. Then there exists a vertex $\sigma'$ in the same component of $\FS(X,Y)$ as $\sigma$ such that $\sigma'(x) = y$. 
\end{lemma}
\begin{proof}[Proof of Proposition~\ref{p1}]
 We  proceed by induction on $n$. Note that $n\geq k$. If $n=k$, then Proposition \ref{p1} holds by Theorem \ref{complete}. Assume that Proposition \ref{p1} is true for $n-1$. We now show it also holds for $n$. 
 
 Fix an arbitrary vertex $y_0\in V(Y)$ and a vertex $\sigma_0$ of $\FS(\sf{Lollipop}_{n-k,k},Y)$ satisfying  $\sigma_0(1)=y_0$. We claim that for any other vertex $\sigma$ of  $\FS(\sf{Lollipop}_{n-k,k},Y)$, $\sigma$ is in the same component of $\FS(\sf{Lollipop}_{n-k,k},Y)$ as $\sigma_0$, which implies that $\FS(\sf{Lollipop}_{n-k,k},Y)$ is connected.

 Let $\sigma$  any vertex of  $\FS(\sf{Lollipop}_{n-k,k},Y)$ other than $\sigma_0$.
 Apply Lemma \ref{any} on $\FS(X,Y)$ with $X=\sf{Lollipop}_{n-k,k}$, $x=1$, $y=y_0$  and $\sigma$, then there exists a vertex $\sigma'$ in the same component of $\FS(\sf{Lollipop}_{n-k,k},Y)$ as $\sigma$ such that $\sigma'(1)=y_0=\sigma_0(1)$. The graph $\sf{Lollipop}_{n-k,k}|_{\{2,3,\dots,n\}}$ is isomorphic to $\sf{Lollipop}_{n-1-k,k}$ and the graph $Y|_{V(Y)\backslash y_0}$ satisfies the property that every $k$-vertices induced subgraph of $Y|_{V(Y)\backslash y_0}$ is connected. By the induction hypothesis, $\FS(\sf{Lollipop}_{n-k,k}|_{\{2,3,\dots,n\}},Y|_{V(Y)\backslash y_0})$ is connected, which guarantees that $\sigma'$ and $\sigma_0$ are in the same component of $\FS(\sf{Lollipop}_{n-k,k},Y)$, i.e., $\sigma$ and  $\sigma_0$ are in the same component of $\FS(\sf{Lollipop}_{n-k,k},Y)$.
\end{proof}

\subsection{Disconnectedness}

\begin{proposition}\label{p2}
 Let $Y$ be a graph on $n$ vertices. If there exists a disconnected induced subgraph $Y_0$ of $Y$ with $k$ vertices, then $\FS(\sf{Lollipop}_{n-k,k},Y)$ is disconnected.
\end{proposition}

\begin{proof}
 Let  $V(Y_0)=A\cup B$ such that $A\cap B= \varnothing$ and there are no edges between $A$ and $B$.  Let $X=\sf{Lollipop}_{n-k,k}$. We say a vertex $\sigma$  of $\FS(X,Y)$ is {\it special} if there exists an $a_0\in A$ such that $\sigma^{-1}(a_0)\in [n-k+1]$ and $\sigma^{-1}(B)=\{\sigma^{-1}(y)~|~y\in B\}\cap \{1,2,\dots,\sigma^{-1}(a_0)\}=\varnothing$. Such a vertex $a_0\in A$ is called a {\it timid} vertex for $\sigma$. It is easy to see that there exist both special vertex and non-special vertex in $\FS(X,Y)$ because $\sigma'$ is special if $\sigma'(1)\in A$ and $\sigma'$ is non-special if $\sigma'(1)\in B$. 
 We claim that any special vertex is not adjacent to any non-special vertex, which implies that $\FS(\sf{Lollipop}_{n-k,k},Y)$ is disconnected.

 Suppose to the contrary that a special vertex $\sigma$ is adjacent to a non-special vertex $\tau$, then there must be two adjacent vertices $a,c$ in $V(X)=[n]$ such that $\tau=\sigma \circ (a \ c)$, where $(a \ c)$ denotes the transposition of $a,c$ on $[n]$ that swaps the numbers $a$ and $c$.  If none of $a, c$ is $\sigma^{-1}(a_0)$, then $\tau$ is also special with its timid vertex $a_0$. So we may assume that $a=\sigma^{-1}(a_0)$, which implies $\sigma(c)\notin  B$ since $\sigma(a)$ are adjacent to $\sigma(c)$ in $Y$, but there are no edges between $A$ and $B$. In addition, $\sigma(c)$ does not belong to $A$  and $c$ does not equal to $a-1$ since otherwise $\tau$  will still be special, with the timid vertices $c$ and $a_0$ for $\tau$, respectively. We further claim that $a$ equals to $n-k+1$ and $\sigma^{-1}(A)\cap [n-k]=\varnothing$.

Suppose that $a \neq n-k+1$, that is, $a\in [n-k]$, then $c$ equals to $a+1$.  We have $\sigma^{-1}(B)\cap \{1,2,\dots,\tau^{-1}(a_0) \}=\varnothing$ since $\sigma^{-1}(B)\cap \{1,2, \dots, c\}=\varnothing$, i.e., $a_0$ is the timid vertex for $\tau$. So we conclude that $a=n-k+1$. If there is a $\sigma^{-1}(a_0') \in \sigma^{-1}(A)\cap [n-k]$, then we have $\sigma^{-1}(B)\cap \{1,2,\dots,\tau^{-1}(a_0') \}=\varnothing$ since $\tau^{-1}(a_0')=\sigma^{-1}(a_0')$, and so $a_0'$  is the timid vertex for $\tau$.
 
 The final contradiction arises since all the $k+1$ vertices in the set $\{\tau^{-1}(c)\} \cup \tau^{-1}(A)\cup \tau^{-1}(B)$ are contained in $[n-k+1,n]$, which is a set with only $k$ elements.
\end{proof}

Combining Propositions \ref{p1} and \ref{p2},  we complete the proof of Theorem \ref{main}.

\vskip 2mm
\noindent{\bf Remark.}  Let $\lambda_1\geq\lambda_2\geq\cdots\geq\lambda_k$ be integers. The \emph{spider} $\Spider(\lambda_1,\lambda_2,\ldots,\lambda_{k})$ is a graph on $n=1+\sum_{i=1}^k \lambda_i$ vertices obtained by connected one of the two ends of each path of order $\lambda_1, \lambda_2, \dots\, \lambda_k$ to a new common vertex.
Defant, Dong, Lee and Wei \cite{DDLW} showed the following.

\begin{theorem}\cite{DDLW}\label{spider}
Let $\lambda_1\geq\cdots\geq\lambda_k$ be positive integers and  $n=\lambda_1+\cdots+\lambda_k+1$. Let $Y$ be a graph on $n$ vertices. If there exists a disconnected induced subgraph of $Y$ with $n-\lambda_1$ vertices, then $\FS(\Spider(\lambda_1,\ldots,\lambda_k),Y)$ is disconnected.  
\end{theorem}

\noindent We can see that Proposition \ref{p2}  strengthens  Theorem \ref{spider} since $\sf{Spider}(\lambda_1,\ldots,\lambda_k)$ is isomorphic to a spanning subgraph of $\sf{Lollipop}_{\lambda_1,n-\lambda_1}$, where $n=\lambda_1+\cdots+\lambda_k+1$.

\section{Open problem}
Let $2\le k \le n$ be integers. The {\it dandelion} graph $\mathsf{Dand}_{n-k,k}$ is a spider of order $n$ with parameters $\lambda_1=n-k$ and $\lambda_2=\dots=\lambda_k=1$, that is, $\mathsf{Dand}_{n-k,k}$ is obtained by identifying one end of a $\sf{Path}_{n-k+1}$ with the center of a $\sf{Star}_k$.
It is clear that $\mathsf{Dand}_{n-2,2}$ is precisely $\mathsf{Lollipop}_{n-2,2}$, and $\mathsf{Dand}_{n-k,k}$ is a proper spanning subgraph of $\mathsf{Lollipop}_{n-k,k}$ for $k\ge 3$.  Defant and Kravitz \cite{DK} showed that $\FS(\mathsf{Lollipop}_{n-3,3},Y)$ is connected if and only if $\FS(\mathsf{Dand}_{n-3,3},Y)$ is connected. This leads us to ask when the edges  not in the spanning subgraph $\mathsf{Dand}_{n-k,k}$ of $\mathsf{Lollipop}_{n-k,k}$ are not  necessary for the connectedness of $\FS(\mathsf{Lollipop}_{n-k,k},Y)$. More precisely, we have the following.

  \begin{problem}\label{p}
 For what  $k$ and $n$, it holds that $\FS(\mathsf{Lollipop}_{n-k,k},Y)$ is connected if and only if $\FS(\mathsf{Dand}_{n-k,k},Y)$ is connected?
\end{problem}



\noindent By Theorem \ref{main} and a result due to Defant, Dong, Lee and Wei \cite{DDLW}, one can see that  the statement holds   for $n\ge 2k-1$.  On the other hand,  by Theorem \ref{complete} and a result of Wilson \cite{W}, the statement is false for $n=k$. 


\section*{Acknowledgments}

 This research was supported by NSFC under grant numbers  12161141003 and 11931006.



\begin{thebibliography}{99}
\small \setlength{\itemsep}{-.10mm}



\bibitem{ADK} \href{https://www.sciencedirect.com/science/article/abs/pii/S0095895622000259?via%3Dihub} {N. Alon, C. Defant, N. Kravitz, Typical and extremal aspects of friends-and-strangers graphs, J. Combin. Theory Ser. B (2022).}




\bibitem{B} \href{https://www.sciencedirect.com/science/article/abs/pii/S0195669822000257}{K. Bangachev, On the asymmetric generalizations of two extremal questions on friends-and strangers graphs, Eur. J. Combin. 104 (2022) 103529.}
 

\bibitem{DDLW}  \href{https://arxiv.org/pdf/2209.01704.pdf}{C. Defant, D. Dong, A. Lee, M. Wei, Connectedness and cycle spaces of friends-and-strangers graphs, ArXiv preprint, arXiv:2209.01704, 2022.}

\bibitem{DK}  \href{http://refhub.elsevier.com/S0095-8956(22)00025-9/bib28F20A02BF8A021FAB4FCEC48AFB584Es1}{C. Defant, N. Kravitz, Friends and strangers walking on graphs, Combin. Theory 1 (2021).}
\bibitem{GR} \href{https://scholar.google.com/scholar?hl=zh-CN&as_sdt=0%2C5&q=C.+Godsil%2C+G.+Royle%2C+Algebraic+Graph+Theory&btnG=}{C. Godsil, G. Royle, Algebraic Graph Theory, Springer, 2001.} 
\bibitem{L}  \href{https://arxiv.org/abs/2210.04768}{A. Lee, Connectedness in friends-and-strangers graphs of spiders and complements, arXiv:2210.04768, 2022.}

\bibitem{M}  \href{https://arxiv.org/abs/2210.03864}{A. Milojevic, Connectivity of old and new models of friends-and-strangers graphs, arXiv:2210.03864, 2022.}


\bibitem{J1}  \href{https://arxiv.org/abs/2201.00665}{R. Jeong, Diameters of connected components of friends-and-strangers graphs are not polynomially bounded, ArXiv preprint, arXiv:2201.00665v3, 2022.}

\bibitem{J2} \href{https://arxiv.org/abs/2203.10337}{ R. Jeong, On structural aspects of friends-and-strangers graphs, ArXiv preprint, arXiv:2203.10337v1, 2022.}


 \bibitem{Wang} \href{https://arxiv.org/abs/2208.00801}{L. Wang and Y. Chen, Connectivity of friends-and-strangers graphs on random pairs, arXiv:2208.00801, 2022.}

 
  \bibitem{W} \href{http://refhub.elsevier.com/S0095-8956(22)00025-9/bibABD7372BBA55577590736EF6CB3533C6s1}{ R.M. Wilson, Graph puzzles, homotopy, and the alternating group, J. Combin. Theory, Ser. B 16
(1974) 86-96.}
 

\end{thebibliography}
\end{document}